\documentclass[a4paper,11pt,reqno]{article}

\usepackage{amsmath,amssymb,eucal,amsthm}
\DeclareMathOperator{\Li}{\mathrm{Li}}
\DeclareMathOperator{\wt}{\mathrm{w}}
\DeclareMathOperator{\dep}{\mathrm{d}}
\DeclareMathOperator{\h}{\mathrm{h}}
\def\aa{\mathbf{a}}
\def\bb{\mathbf{b}}
\def\kk{\mathbf{k}}
\def\={\,=\,}

\def\R{\mathbb{R}}

\font\fivecy=wncyr5  \def\sa{{\hbox{\fivecy X}}}

\newtheorem{theorem}{Theorem}[section]

\theoremstyle{definition}

\begin{document}

\title{On multiple zeta values of extremal height}

\author{Masanobu Kaneko and Mika Sakata}

\date{\today}

\maketitle

\begin{abstract}
We give three identities involving multiple zeta values of height one and of maximal height; an explicit 
formula for the height-one multiple zeta values, a regularized sum formula, and a sum formula for the multiple zeta values of maximal height.
\end{abstract}

\section{Main results} \label{sec-1}

The multiple zeta value (MZV) is a real number given by the nested series
\[ \zeta(k_1,\ldots,k_r)\=\sum_{0<m_1<\cdots<m_r}\frac1{m_1^{k_1}\cdots m_r^{k_r}} \]
for each index set $\kk=(k_1,\ldots,k_r)$ of positive integers $k_i$, with the last entry $k_r>1$ 
for convergence. The quantities $\wt(\kk):=k_1+\cdots+k_r$, $\dep(\kk):=r$,
and $\h(\kk):=\#\{i\,|\,k_i>1, 1\le i\le r\}$ are called respectively the weight, the depth, and 
the height of the index set $\kk$ (or of the multiple zeta value $\zeta(\kk)=\zeta(k_1,\ldots,k_r)$).

In this paper, we present the following three identities which involve multiple zeta values of extremal height,
that is,  the MZVs of height one or of maximal height (all components of the index are greater than one). 

\begin{theorem}[Explicit formula for the height-one MZV]\label{thm1}
For any integers $r,k\geq 1$, we have 
\begin{equation}\label{symmzv}
\zeta(\underbrace{1, \ldots , 1}_{r-1}, k+1)
\=\sum^{\min(r,k)}_{j=1}(-1)^{j-1}\sum_{\substack{\wt(\aa)=k,\,\wt(\bb)=r\\\dep(\aa)=\dep(\bb)=j}}
\zeta(\aa+\bb),
\end{equation}
where,  for two indices $\aa=(a_1,\ldots,a_j)$ and $\bb=(b_1,\ldots,b_j)$ of the same depth,
$\zeta(\aa+\bb)$ denotes $\zeta(a_1+b_1,\ldots,a_j+b_j)$. 
\end{theorem}
Note that the right-hand side of this formula is symmetric in $r$ and $k$, and thus the formula 
makes the duality $\zeta(\underbrace{1, \ldots , 1}_{r-1}, k+1)=\zeta(\underbrace{1, \ldots , 1}_{k-1}, r+1)$
visible.  (N.B. We use the duality in our proof, so that we are not giving an alternative proof
of the duality.)  To our knowledge, no such symmetric explicit formula for the height-one MZV 
has been known, except for the well-known symmetric generating function \cite{Ao, Dr}:
\[ 1-\sum_{r,k\ge1}  \zeta(\underbrace{1, \ldots , 1}_{r-1}, k+1) x^ry^k\=\frac{\Gamma(1-x)\Gamma(1-y)}{\Gamma(1-x-y)}
\= \exp\biggl(\sum_{n=2}^\infty \zeta(n)\frac{x^n+y^n-(x+y)^n}{n}\biggr). \]
Also, we should remark that the right-hand side of the theorem is symmetric with respect to any permutations of the 
arguments, so that the theorem of Hoffman \cite[Theorem~2.2]{Hof} ensures the right-hand side is a
polynomial in the Riemann zeta values $\zeta(n)$, the fact also can be seen from the generating function above.
Moreover, we note that all the MZVs appearing on the right-hand side is of maximal height.

As a final remark, the case of $r=2$ gives nothing but the ``sum formula'' for depth 2 ($r=1$ gives the trivial
identity $\zeta(k+2)=\zeta(k+2)$). It was H.~Tsumura 
who first remarked that we could obtain the depth 2 sum formula if we looked at the behavior at $s=0$  of 
the identity \eqref{keyid} in the next section for $r=2$.

Recall the classical sum formula states that the sum of all
MZVs of fixed weight and depth is equal to the Riemann zeta value of that weight.  If we extend the sum
to include non-convergent MZVs with the shuffle regularization, the result will be the height-one MZV (up to sign).
 
\begin{theorem}[Shuffle-regularized sum formula] \label{thm2}
For any integers $r,k\ge1$, we have
\[ \sum_{\substack{\wt(\kk)=r+k\\\dep(\kk)=r}} \zeta^\sa(\kk)=(-1)^{r-1}\zeta(\underbrace{1,\ldots,1}_{r-1},k+1), \]
where $\zeta^\sa(\kk)$ is the shuffle regularized value which will be recalled in \S2.
\end{theorem}
\noindent 
We do not know if there exists any nice stuffle-regularized sum formula.

Finally, we give a kind of sum formula for the maximal-height MZVs in the form of generating function.
This is essentially known, but may be new in this form of presentation.  
Let $T(k)$ be the sum of all multiple zeta values of weight $k$ and of maximal height:
\[ T(k):=\sum_{\substack{k_1+\cdots+k_r=k\\
r\ge1,\,{}^\forall k_i\ge2}}\zeta(k_1,\ldots,k_r). \]
Recall the multiple zeta-star value $\zeta^\star(k_1,\ldots,k_r)$ is given by the non-strict nested sum
\[ \zeta^\star(k_1,\ldots,k_r)\=\sum_{0<m_1\le\cdots\le m_r}\frac1{m_1^{k_1}\cdots m_r^{k_r}}. \]

\begin{theorem}\label{thm3}
We have the generating series identity
\[ 1+\sum_{k=2}^\infty T(k)x^k=\biggl(1+\sum_{n=1}^\infty \zeta^\star(\underbrace{2, \ldots ,2}_n)x^{2n}\biggr)
\biggl(1+\sum_{n=1}^\infty \zeta(\underbrace{3, \ldots ,3}_n)x^{3n}\biggr). \]
\end{theorem}

After some necessary preliminaries in the next section, we prove these results in \S3.

\section{Preliminaries}
Recall the function introduced in \cite{AK1999},
\begin{equation}
\xi(k_1,\ldots,k_r;s)=\frac{1}{\Gamma(s)}\int_0^\infty {t^{s-1}}\frac{\Li_{k_1,\ldots,k_r}(1-e^{-t})}{e^t-1}\,dt,  \label{xidef}
\end{equation}
where $\Li_{k_1,\ldots,k_r}(z)$ is the multiple polylogarithm function defined by
\[  \Li_{k_1,\ldots,k_r}(z)\= \sum_{0<m_1<\cdots<m_r}\frac{z^{m_r}}{m_1^{k_1}\cdots m_r^{k_r}}.\]
When $k_r>1$, the value at $z=1$ of $\Li_{k_1,\ldots,k_r}(z)$ is nothing but the multiple zeta value
$\zeta(k_1,\ldots,k_r)$.  The function $\xi(k_1,\ldots,k_r;s)$ is analytically continued to an {\it entire} function in $s$. 
In the special case where $(k_1,\ldots,k_r)=(\underbrace{1,\ldots,1}_{r-1},k)$, Arakawa and the first-named author 
have established in \cite[Theorem 8]{AK1999}
the following identity (we interchange $r$ and $k$
and shift $s$ to $s+1$), which is crucial in our proofs of Theorems~\ref{thm1} and \ref{thm2}:
\begin{align} \label{keyid}
\xi(\underbrace{1, \ldots , 1}_{k-1}, r ; s+1)
=&\ (-1)^{r-1}\sum_{\substack{a_{1}+\cdots +a_{r}=k\\{}^\forall a_p\geq 0}} \binom{s+a_{r}}{a_{r}}
\,\zeta(a_{1}+1, \ldots , a_{r-1}+1,a_{r}+1+s)\\ \nonumber
&+\sum^{r-2}_{i=0}(-1)^{i}\zeta(\underbrace{1, \ldots , 1}_{k-1}, r-i)\,\zeta(\underbrace{1, \ldots ,1}_{i},1+s), 
\end{align}
for any $r, k\geq 1$. Here, we have introduced a complex variable $s$ in the outer-most exponent of the MZV;
\[ \zeta(k_1,\ldots,k_{r-1},k_r+s)\,:=\, \sum_{0<m_1<\cdots<m_r}\frac1{m_1^{k_1}\cdots m_{r-1}^{k_{r-1}}m_r^{k_r+s}}. \]
As remarked in \cite[Remark~3.7]{KT}, equation~\eqref{keyid} is equivalent to the connection formula
of Euler's type of the multi-polylogarithm $\Li_{\scriptsize{\underbrace{1,\ldots,1}_{k-1},r}}(z)$.
It is shown in \cite{AK1999} that the function $\zeta(k_1,\ldots,k_{r-1},k_r+s)$ can be meromorphically continued to the whole $s$-plane,
and has a pole at $s=0$ if $k_r=1$.  We need the description of the principal part at $s=0$ in terms of regularized polynomials,
which we now explain. 

For an index $\kk=(k_1,\ldots,k_r)$, we denote by $Z_\kk^\sa(T)$ and $Z_\kk^\ast(T)$
respectively the shuffle and the stuffle (harmonic) regularized polynomial associated to $\kk$.
These are the polynomials in $\R[T]$ uniquely characterized by the asymptotics
\[ \Li_{k_1,\ldots,k_r}(z)\=Z_\kk^\sa(-\log(1-z))+O((1-z)^\varepsilon)\quad\text{as}\ z\to1\ \text{for some}\ \varepsilon>0 \]
and
\[ \sum_{0<m_1<\cdots<m_r<M}\frac1{m_1^{k_1}\cdots m_{r}^{k_{r}}}\=Z_\kk^\ast(\log M+\gamma)+O(M^{-\varepsilon})
\quad\text{as}\ M\to\infty\ \text{for some}\ \varepsilon>0,\]
where $\gamma$ is Euler's constant. 
We refer the reader to \cite{IKZ} for details about the regularizations.  We denote the constant term $Z_\kk^\sa(0)$
of the shuffle-regularized polynomial $Z_\kk^\sa(T)$ by $\zeta^\sa(\kk)$ and call it the shuffle-regularized value of
(possibly divergent) $\zeta(\kk)$.
If $\kk$ is of the 
form $\kk=(k_1,\ldots,k_n,\underbrace{1,\ldots,1}_m)$ with $k_n>1, m\ge0$, then both 
$Z_\kk^\sa(T)$ and $Z_\kk^\ast(T)$ are of degree $m$ and each coefficient of $T^i$ is 
a linear combination of multiple zeta values of weight $m-i$.  If $m=0$ (and so $n=r$), then $Z_\kk^\sa(T)=Z_\kk^\ast(T)=
Z_\kk^\sa(0)=Z_\kk^\ast(0)=\zeta(k_1,\ldots,k_r)$.
Now write 
\[ Z_\kk^\sa(T)\=\sum_{i=0}^m a_i(\kk)\frac{T^i}{i!}\quad \text{and}\quad  Z_\kk^\ast(T)\=
\sum_{i=0}^m b_i(\kk)\frac{(T-\gamma)^i}{i!}.\]
Then, as shown in \cite{AK2004}, the principal parts at $s=0$ of 
$\Gamma(s+1)\zeta(k_1,\ldots,k_{r-1},k_r+s)$ and $\zeta(k_1,\ldots,k_{r-1},k_r+s)$ 
are given respectively by
\begin{equation}\label{shpole}
\Gamma(s+1)\zeta(k_1,\ldots,k_{r-1},k_r+s)=\sum_{i=0}^m \frac{a_i(\kk)}{s^i}+O(s)\quad  (s\to0) 
\end{equation}
and
\begin{equation}\label{harpole} \zeta(k_1,\ldots,k_{r-1},k_r+s)=\sum_{i=0}^m \frac{b_i(\kk)}{s^i}+O(s)\quad (s\to0).
\end{equation}
We take this opportunity to point out a flaw in the proof in \cite{AK2004}. The integral in the sum
on the right of the equation below (32) may not converge. But the argument can easily be modified by
splitting the integral $\int_0^\infty$ on the left as $\int_0^1+\int_1^\infty$ and looking at the limits when $s\to 0$ separately.
 
\section{Proofs}
\begin{proof}[Proof of Theorem~\ref{thm1}]
Since we have the duality $\zeta(\underbrace{1, \ldots , 1}_{r-1}, k+1)=\zeta(\underbrace{1, \ldots , 1}_{k-1}, r+1)$
and the right-hand side of \eqref{symmzv} 
is symmetric in $r$ and $k$, it is enough to prove the theorem under the assumption $k\geq r$. 
We proceed by induction on $r$. When $r=1$, both sides become $\zeta(k+1)$ and the assertion is true for all
$k\ge1$.  Suppose $r\ge2$ and the theorem is true when the depth on the left is less than $r$ (and $k$ is greater than
or equal to the depth).  

We look at the values at $s=0$ of both sides of \eqref{keyid}.  The value $\xi(\underbrace{1, \ldots , 1}_{k-1}, r ; 1)$ on the 
left is evaluated in \cite[Theorem  9]{AK1999} and is equal to $\zeta(\underbrace{1, \ldots , 1}_{r-1}, k+1)$.  
Since the functions $\zeta(a_{1}+1, \ldots , a_{r-1}+1,a_{r}+1+s)$  with $a_r=0$ as well as 
$\zeta(\underbrace{1, \ldots ,1}_{i},1+s)$ on the right have poles at $s=0$, we need to look at the constant term 
of the Laurent expansion of the right-hand side. (Because $\xi(\underbrace{1, \ldots , 1}_{k-1}, r ; s+1)$ is 
entire, all the poles on the right actually cancel out.)  In what follows within the proof of Theorem~\ref{thm1},
we simply write the constant term at $s=0$ of $\zeta(k_1,\ldots,k_{r-1},k_r+s)$ as $\zeta(k_1,\ldots,k_{r-1},k_r)$
even when $k_r=1$,
which is equal to $Z_{k_1,\ldots,k_r}^\ast(\gamma)$ as recalled in the previous section.  
Note that these values satisfy the stuffle (harmonic) product rule.  With this convention, we have
\begin{align*}
\zeta(\underbrace{1, \ldots , 1}_{r-1}, k+1)
=\ & (-1)^{r-1}\sum_{\substack{a_{1}+\cdots +a_{r}=k\\^{\forall}a_p\geq 0}}
\zeta(a_{1}+1, \ldots , a_{r}+1)\\
&+\sum^{r-2}_{i=0}(-1)^{i}\zeta(\underbrace{1, \ldots , 1}_{k-1}, r-i)\cdot\zeta(\underbrace{1, \ldots ,1}_{i+1}). 
\end{align*}
We apply the duality $\zeta(\underbrace{1,\ldots 1}_{k-1}, r-i)=\zeta(\underbrace{1,\ldots , 1}_{r-i-2}, k+1)$
in the second sum on the right and use the induction hypothesis (since $r-i-1<r$) to obtain
\begin{align*}
 \zeta(\underbrace{1, \ldots , 1}_{r-1}, k+1)
&= (-1)^{r-1}\sum_{\substack{a_{1}+\cdots +a_{r}=k\\^{\forall}a_p\geq 0}}
\zeta(a_{1}+1, \ldots , a_{r}+1)\\
&\quad  +\sum^{r-2}_{i=0}(-1)^i\sum^{r-i-1}_{j=1}(-1)^{j-1}\sum_{\substack{\wt(\aa)=k,\,\wt(\bb)=r-i-1\\\dep(\aa)=\dep(\bb)=j}}
\zeta(\aa+\bb)\cdot \zeta(\underbrace{1, \ldots ,1}_{i+1})\\
&= (-1)^{r-1}\sum_{\substack{a_{1}+\cdots +a_{r}=k\\^{\forall}a_p\geq 0}}
\zeta(a_{1}+1, \ldots , a_{r}+1)\\
&\quad  +\sum^{r-1}_{j=1}(-1)^{j-1}\sum_{\substack{\wt(\aa)=k\\ \dep(\aa)=j}}\sum^{r-j-1}_{i=0}(-1)^i
\sum_{\substack{\wt(\bb)=r-i-1\\ \dep(\bb)=j}}
\zeta(\aa+\bb)\cdot \zeta(\underbrace{1, \ldots ,1}_{i+1}). 
\end{align*}
Now we expand the product $\zeta(\aa+\bb)\cdot \zeta(\underbrace{1, \ldots ,1}_{i+1})$ by using the 
stuffle product and re-arrange the terms according to the number of $1$'s to compute the inner
sum \[\sum^{r-j-1}_{i=0}(-1)^i\sum_{\substack{\wt(\bb)=r-i-1\\ \dep(\bb)=j}}
\zeta(\aa+\bb)\cdot \zeta(\underbrace{1, \ldots ,1}_{i+1}).\] For that purpose, we introduce another notation.
For a fixed index $\aa=(a_1,\ldots,a_j)$ of depth $j$ and integers $l, n\ge0$, we set
\[S(\aa,l,n):=\sum_{\substack{\wt(\bb)=r-l\\\dep(\bb)=j,\,\h(\bb)=n}}
\zeta(a_1+b_1,\ldots,1,\ldots,a_s+b_s,\ldots,1,\ldots,a_j+b_j),\]
where the sum runs over all $\bb=(b_1,\ldots,b_j)$ of weight $r-l$, depth $j$, and height $n$, and over all possible 
positions of exactly $l$ $1$'s in the arguments.  
Then, by the stuffle product rule, we have 
\[ \sum_{\substack{\wt(\bb)=r-i-1\\ \dep(\bb)=j}}\zeta(\aa+\bb)\cdot \zeta(\underbrace{1, \ldots ,1}_{i+1})
=\sum_{l=\max(0,i+1-j)}^{i+1} \sum_{n=i+1-l}^j\binom{n}{i+1-l}S(\aa,l,n).\]
We note that,  when we expand $\zeta(\aa+\bb) \zeta(\underbrace{1, \ldots ,1}_{i+1})$ by the stuffle product,
the number of 1's in each term should at least $i+1-j$ when $j<i+1$. And if the number of 1's is $l$,
then the height $n$ on the right varies from $i+1-l$ to $j$.   A particular term in the sum $S(\aa,l,n)$
on the right comes in exactly $\binom{n}{i+1-l}$ ways from the product $\zeta(\aa+\bb) \zeta(\underbrace{1, \ldots ,1}_{i+1})$
on the left, because there are ${i+1-l}$ out of $n$ positions of the index $\aa+\bb$ on the left which
produces that particular term on the right by colliding $i+1-l$ 1's at those positions.

   When we sum this up alternatingly
for $i=0,\ldots, r-j-1$ with signs, all coefficients of $S(\aa,l,n)$ with $n,l\ge1$ vanish, because of
the binomial identity $\sum_{i=l-1}^{n+l-1}(-1)^i\binom{n}{i+1-l}=0$ if $n,l\ge1$.  Hence, also by the identity
$\sum_{i=0}^{n-1}(-1)^i\binom{n}{i+1}=1$ if $n\ge1$ (the case $l=0$), we obtain 
\[ \sum^{r-j-1}_{i=0}(-1)^i\sum_{\substack{\wt(\bb)=r-i-1\\ \dep(\bb)=j}}
\zeta(\aa+\bb)\cdot \zeta(\underbrace{1, \ldots ,1}_{i+1})=\sum_{n=1}^jS(\aa,0,n)+(-1)^{r-j-1}S(\aa,r-j,0).\]
When $j\le r-1$, we have $\sum_{n=1}^jS(\aa,0,n)=\sum_{\wt(\bb)=r,\,\dep(\bb)=j}\zeta(\aa+\bb)$ and this 
gives 
\begin{equation}\label{part1}
\sum_{j=1}^{r-1}(-1)^{j-1}\sum_{\substack{\wt(\aa)=k,\,\wt(\bb)=r\\\dep(\aa)=\dep(\bb)=j}}\zeta(\aa+\bb).
\end{equation}
Finally, we have 
\begin{align*}
&\sum_{j=1}^{r-1}(-1)^{j-1}\sum_{\substack{\wt(\aa)=k\\ \dep(\aa)=j}}(-1)^{r-j-1}S(\aa,r-j,0)\\
&= (-1)^r \sum_{j=1}^{r-1}\sum_{\substack{\wt(\aa)=k\\ \dep(\aa)=j}}S(\aa,r-j,0)\\
&=(-1)^r \sum_{\substack{a_1+\cdots+a_r=k\\ a_p\ge0,\,\text{at least one }a_p=0}}
\zeta(a_1+1,\ldots,a_r+1).
\end{align*}
Hence, this and the terms in 
\[ (-1)^{r-1}\sum_{\substack{a_{1}+\cdots +a_{r}=k\\^{\forall}a_p\geq 0}}
\zeta(a_{1}+1, \ldots , a_{r}+1)\] 
with at least one $a_p=0$ cancel out, thereby remains the term 
\begin{equation}\label{part2} (-1)^{r-1}\sum_{\substack{\wt(\aa)=k,\,\wt(\bb)=r\\\dep(\aa)=\dep(\bb)=r}}\zeta(\aa+\bb).
\end{equation}
The sum of \eqref{part1} and \eqref{part2} gives the right-hand side of the theorem, and our proof is done. 
\end{proof}
\bigskip

\begin{proof}[Proof of Theorem~\ref{thm2}]
We multiply $\Gamma(s+1)$ on both sides of the identity \eqref{keyid} and look at the 
constant terms of the Laurent expansions at $s=0$.  The left-hand side is holomorphic at $s=0$ and gives the 
value $\zeta(\underbrace{1,\ldots,1}_{r-1},k+1)$ as we already saw in the last subsection.
The function $\binom{s+a_r}{a_r}\Gamma(s+1) \zeta(a_1+1,\ldots,a_{r-1}+1,a_r+1+s)$ on the right is holomorphic
at $s=0$ if $a_r>1$, and in that case gives the value $\zeta(a_1+1,\ldots,a_{r-1}+1,a_r+1)$.
If $a_r=0$, then $\binom{s+a_r}{a_r}\Gamma(s+1) \zeta(a_1+1,\ldots,a_{r-1}+1,a_r+1+s)=
\Gamma(s+1) \zeta(a_1+1,\ldots,a_{r-1}+1,1+s)$ has a pole at $s=0$ and its constant term of the 
Laurent expansion is $\zeta^\sa(a_1+1,\ldots,a_r+1)$ by \eqref{shpole}.  
On the other hand, the function $\Gamma(s+1)\zeta(\underbrace{1,\ldots,1}_i,1+s)$ has no constant term
at $s=0$ because $Z_{\scriptsize{\underbrace{1,\ldots,1}_{i+1}}}^\sa(T)=T^{i+1}/(i+1)!$, and hence 
we conclude the proof of the theorem. 
\end{proof}

We remark that we can prove the theorem alternatively by computing directly the left-hand side 
using the regularization formula \cite[(5.2)]{IKZ}. Also, by Theorem~\ref{thm2} and \cite[Corollary~5]{IKZ}, we easily 
obtain the following sum formula for the shuffle-regularized polynomials:
\[ \sum_{\substack{\wt(\kk)=r+k\\\dep(\kk)=r}} 
\zeta^\sa(\kk;T)=\sum_{i=0}^{r-1} (-1)^{r-1-i}\zeta(\underbrace{1,\ldots,1}_{r-1-i},k+1)
\,\frac{T^i}{i!} \]
for any $r,k\ge1$, where $\zeta^\sa(\kk;T)=Z_{\R}^\sa(w)$ in the notation of \cite{IKZ} with
$w$ being a word corresponding to $\kk$.

\bigskip

\begin{proof}[Proof of Theorem~\ref{thm3}]
This is almost obvious if we write $k_i$ ($\ge2$) as $k_i=2+\cdots+2$ ($k_i$\,: even) or $k_i=3+2+\cdots+2$
 ($k_i$\,: odd), and consider the stuffle product of  $\zeta^\star(2, \ldots ,2)\zeta(3, \ldots ,3)$ after
 writing $\zeta^\star(2, \ldots ,2)$ as sums of ordinary multiple zeta values. 
 
 An alternative proof is given by using the main identity in \cite{OZ}.  As is already remarked there, if we specialize
 $y=0$ and $z=x^2$ in equation (3) in \cite{OZ}, we obtain
 \[ 1+\sum_{k=2}^\infty T(k)x^k=\exp\biggl(\sum_{n=1}^\infty\frac{\zeta(2n)}{n}x^{2n}\biggr)
 \cdot \exp\biggl(\sum_{n=1}^\infty(-1)^{n-1}\frac{\zeta(3n)}{n}x^{3n}\biggr).\] 
 It is standard that 
\[\exp\biggl(\sum_{n=1}^\infty\frac{\zeta(2n)}{n}x^{2n}\biggr)=\Gamma(1+x)\Gamma(1-x)
=\prod_{m=1}^\infty \biggl(1-\frac{x^2}{m^2}\biggr)^{-1}
=1+\sum_{n=1}^\infty \zeta^\star(\underbrace{2, \ldots ,2}_n)x^{2n}, \]
whereas the identity 
\[ \exp\biggl(\sum_{n=1}^\infty(-1)^{n-1}\frac{\zeta(3n)}{n}x^{3n}\biggr)=
1+\sum_{n=1}^\infty \zeta(\underbrace{3, \ldots ,3}_n)x^{3n} \]
is a special case of \cite[Corollary 2 of Proposition 4]{IKZ}.
\end{proof}

\section*{Acknowledgements}
The authors thank Hirofumi Tsumura for his suggestion to look at the identity \eqref{keyid} more
closely, that lead us to Theorem~\ref{thm1}.  They also thank S.~Yamamoto for his calling our
attention to the reference \cite{OZ} in relation to Theorem~\ref{thm3}.  This paper was written
during the first author's stay at the I.H.E.S., France.  He thanks Francis Brown for the invitation
and for all the hospitality and support given there.
This work is supported by Japan Society for the Promotion of Science, 
Grant-in-Aid for Scientific Research (B) 23340010 (M.K.) and Grant-in-Aid for JSPS Fellows 14J00005 (M.S.).

\ 

\ 

\begin{flushleft}
\begin{small}
{M.~Kaneko}: 
{Faculty of Mathematics, 
Kyushu University, 
Motooka 744, Nishi-ku,
Fukuoka 819-0395, 
Japan}

e-mail: {\tt mkaneko@math.kyushu-u.ac.jp}

\

{M.~Sakata}: 
{Graduate School of Mathematics, Kyushu University, 
Motooka 744, Nishi-ku,
Fukuoka 819-0395, 
Japan}

e-mail: {\tt m-sakata@math.kyushu-u.ac.jp}
\end{small}
\end{flushleft}


\begin{thebibliography}{999}

\bibitem{Ao} K.~Aomoto, 
Special values of hyperlogarithms and linear difference schemes, 
Illinois J. Math., {\bf 34-2} (1990), 191--216.
 
\bibitem{AK1999}  T.~Arakawa and M.~Kaneko, 
Multiple zeta values, poly-Bernoulli numbers, and related zeta functions,  Nagoya Math. J., {\bf 153} (1999), 189--209.

\bibitem{AK2004}  T.~Arakawa and M.~Kaneko, 
On multiple $L$-values, J. Math. Soc. Japan, {\bf 56} (2004), 967--991.

\bibitem{Dr} V.~G.~Drinfel'd, 
On quasitriangular quasi-Hopf algebras and a group closely connected with Gal($\bar{\mathbb{Q}}/\mathbb{Q}$), 
Leningrad Math. J. {\bf 2} (1991), 829--860.  

\bibitem{Hof}  M.~E.~Hoffman, Multiple harmonic series, Pacific J. Math. {\bf 152} (1992), 275--290.

\bibitem{IKZ}  K.~Ihara, M.~Kaneko and D.~Zagier, 
Derivation and double shuffle relations for multiple zeta values, Compositio Math., {\bf 142} (2006), 307--338.

\bibitem{KT} M.~Kaneko and H.~Tsumura, Multi-poly-Bernoulli numbers and related zeta functions, preprint,
arXiv:1503.02156, 2015.  

\bibitem{OZ}  Y.~Ohno and D.~Zagier, 
Multiple zeta values of fixed weight, depth, and height, 
J. Number Theory, {\bf 74} (1999), 39--43.

\end{thebibliography}
\end{document}